%% file: SFT-Fundamental-ALP.tex
\documentclass{amsart}
\usepackage{amsthm}
\usepackage{amsmath}
\usepackage{amssymb}
\usepackage{euscript}
\usepackage{graphicx}
\usepackage{pgf,tikz}
\usetikzlibrary{arrows}

\newtheorem{theorem}{Theorem}
\newtheorem*{theorem17}{Theorem 17}
\newtheorem*{theorem18}{Theorem 18}
\newtheorem*{theorem28}{Theorem 28}
\newtheorem{corollary}[theorem]{Corollary}

\newtheorem{lemma}[theorem]{Lemma}
\newtheorem{remark}[theorem]{Remark}

\newtheorem{dfn}[theorem]{Definition}

\newcommand{\N}{\mathbb{N}}

\newcommand{\Z}{\mathbb{Z}}

\newcommand{\ns}{\varnothing}

\newcommand{\invlim}[2]{\lim\limits_{\longleftarrow}\{#1,#2\}}

\newcommand{\seq}[1]{\langle #1\rangle}
\newcommand{\toc}{\mathcal{TOC}}
\newcommand{\pdoc}{\mathcal{PDOC}}
\newcommand{\sys}[2]{(#1,#2)}

\begin{document}

\title[SFTs as fundamental objects in the theory of shadowing]{Shifts of finite type as fundamental objects in the theory of shadowing}

\author[C. Good]{Chris Good}
\address[C. Good]{University of Birmingham\\ School of Mathematics\\ Birmingham B15 2TT, UK}
\email[C. Good]{c.good@bham.ac.uk}
\author[J. Meddaugh]{Jonathan Meddaugh}
\address[J. Meddaugh]{University of Birmingham\\ School of Mathematics\\ Birmingham B15 2TT, UK}
\email[J. Meddaugh]{j.meddaugh@bham.ac.uk}
\thanks{The authors gratefully acknowledge support from the European Union through
funding the H2020-MSCA-IF-2014 project ShadOmIC (SEP-210195797)}

%\date{} % Activate to display a given date or no date (if empty),
         % otherwise the current date is printed

\begin{abstract} Shifts of finite type and the notion of shadowing, or pseudo-orbit tracing, are powerful tools in the study of dynamical systems.
	In this paper we prove that there is a deep and fundamental relationship between these two concepts. 

Let $X$ be a compact totally disconnected space and $f:X\to X$ a continuous map.
	We demonstrate that $f$  has shadowing if and only if the system $\sys{f}{X}$ is (conjugate to) the inverse limit of a directed system of shifts of finite type.  In particular, this implies that, in the case that $X$ is the Cantor set, $f$ has shadowing if and only if $(f,X)$ is the inverse limit of a sequence of shifts of finite type. Moreover, in the general compact metric case, where $X$ is not necessarily totally disconnected, we prove that  $f$ has shadowing if and only if $\sys{f}{X}$ is a factor of (i.e.  semi-conjugate to) the inverse limit of a sequence of shifts of finite type by a quotient that almost lifts pseudo-orbits.
\end{abstract}

\maketitle

\section{Introduction}

%The use of symbolic dynamics in the analysis of both discrete and continuous dynamical systems (albeit with slightly different formulations) has a long history, originating well before Hedlund and Morse's first systematic study of symbolic systems for their own sake \cite{morse-hedlund}.  Given a continuous map $f$ on a space $X$ and a finite cover $\{B_1,\dots, B_n\}$ of $X$, one can associate a sequence of symbols $\seq{s_i}$ to a point $x$ so that $s_i=j$ implies $f^i(x)\in B_j$.  Under certain conditions, one can then analyze the discrete system $\sys{f}{X}$ effectively by studying the action of the shift map on the compact metric space of resulting symbol sequences.  If the finite cover is a Markov partition, then the shift space consists of all infinite (or bi-infinite) symbol sequences that do not contain any words from a finite list of forbidden words. In this case, then the shift space is known as a shift of finite type and, if uncountable, is homeomorphic to the Cantor set.

Given a finite set of symbols, a shift of finite type consists of all infinite (or bi-infinite) symbol sequences, which do not contain any of a finite list of forbidden words, under the action of the shift map.  Shifts of finite  type have applications across mathematics, for example in Shannon's theory of information \cite{Shannon} and statistical mechanics.  In particular, they have proved to be a powerful and ubiquitous tool in the study of hyperbolic dynamical systems.  Adler and Weiss \cite{adler-weiss} and Sinai \cite{sinai}, for example,  obtain Markov partitions for hyperbolic automorphisms of the torus and Anosov diffeomorphisms respectively, allowing analysis via shifts of finite type.
Generalising the notion of Anosov diffeomorphisms, Smale \cite{smale-differentiable} isolates subsystems conjugate to shifts of finite type  in certain Axiom A diffeomorphisms. His fundamental example of a horseshoe, conjugate to the full shift space on two symbols, captures the chaotic behaviour of the diffeomorphism on the nonwandering set where the map exhibits hyperbolic behaviour.  Bowen \cite{bowen-markov-partitions} then shows that the nonwandering set of any Axiom A diffeomorphism is a factor of a shift of finite type.  In fact, shifts of finite type appear as horseshoes in many systems both hyperbolic (for example \cite{Smale, williams}) and otherwise \cite{kennedy-yorke}.

For a map $f$ on a metric space $X$, a sequence $\seq{x_i}_{i\in\omega}$ is a $\delta$-pseudo-orbit if $d(f(x_i), x_{i+1})<\delta$.  Pseudo-orbits arise
naturally in the numerical calculation of orbits. It turns out that pseudo-orbits can often be tracked within a specified tolerance by real orbits, in which case $f$ is said to have the shadowing, or pseudo-orbit tracing, property.  Clearly this is of importance when trying to model a system numerically (for example \cite{Corless, Corless2, palmer, Pearson}), especially when the system is expanding and errors might grow exponentially (indeed shadowing follows from expansivity for open maps \cite{PEsakai}, see also \cite{przytycki-urbanski}).
However, shadowing is also of theoretical importance and the notion can be traced back to the analysis of Anosov and Axiom A diffeomorphisms.  Sinai \cite{sinai-measures} isolated subsystems of Anosov diffeomorphisms with shadowing and Bowen \cite{Bowen} proved explicitly that for the larger class of Axiom A diffeomorphisms, the shadowing property holds on the nonwandering set.  However, Bowen \cite{bowen-markov-partitions} had already used shadowing implicitly as a key step in his proof that the nonwandering set of an Axiom A diffeomorphism is a factor of a shift of finite type. The notion of structural stability of a dynamical system was instrumental in the definitions of both Anosov and Axiom A diffeomorphisms \cite{smale-differentiable} and shadowing plays a key role in stability theory \cite{Pil, robinson-stability, walters}.  Shadowing is also key to characterizing omega-limit sets \cite{BGOR-DCDS, Bowen,  MR}. Moreover, fundamental to the current paper is Walters' result \cite{walters} that a shift space has shadowing if and only if it is of finite type.

\medskip

In this paper we prove that there is a deep and fundamental relationship between shadowing and shifts of finite type. It is known that shadowing is generic for homeomorphisms of the Cantor set \cite{bernardes-darji} and that the shifts of finite type form a dense subset of the space of homeomorphisms on the Cantor set \cite{shimomura}. Hirsch \cite{hirsch-expanding} shows that expanding differentiable maps on closed manifolds are factors of the full one sided shift.  In \cite{bowen-markov-partitions}, Bowen considers the induced dynamics on the shift spaces associated with Markov partitions to show that the action of an Axiom A diffeomorphism on its non-wandering set is a factor of a shift of finite type. In this paper, we expand the scope of this type of analysis by considering the actions induced by $f$ on shift spaces associated with several arbitrary finite open covers of the state space $X$, rather than the much more specific Markov partitions.  In doing so, we are able to extend and clarify these results significantly. proving the following.

\begin{theorem17}
	Let $X$ be a compact, totally disconnected Hausdorff space. The map $f:X\to X$ has shadowing if and only if $\sys{f}{X}$ is conjugate to the inverse limit of a directed system of shifts of finite type.
\end{theorem17}

\begin{theorem18} Let $X$ be the Cantor set, or indeed any compact, totally disconnected metric space.
  The map $f:X\to X$  has shadowing if and only if $\sys{f}{X}$ is conjugate to the inverse limit of a sequence of shifts of finite type.
\end{theorem18}

Let $X$ and $Y$ be compact metric spaces and $\phi:X\to Y$ be a factor map (or semiconjugacy) between the systems $f:X\to X$ and $g:Y\to Y$ (so that $\phi\big(f(x)\big)=g\big(\phi(x)\big)$).  We say that $\phi$ almost lifts pseudo-orbits if and only if for all $\epsilon>0$ and $\eta>0$, there exists $\delta>0$ such that for any $\delta$-pseudo-orbit $\seq{y_i}$ in $Y$, there exists an $\eta$-pseudo-orbit $\seq{x_i}$ in $X$ such that $d(\phi(x_i),y_i)<\epsilon$.

\begin{theorem28} Let $X$ be a compact metric space. The map $f:X\to X$ has shadowing if and only if $(f,X)$ is semiconjugate to the inverse limit of a sequence of shifts of finite type by a map which almost lifts pseudo-orbits.
\end{theorem28}

The approach we take is topological rather than metric as this seems to provide the most natural proofs and allows for simple generalization, at least in the zero-dimensional case.  Although we are considering inverse limits of dynamical systems, our techniques are very similar in flavour to the inverse limit of coupled graph covers which have been used by a number of authors to study dynamics on Cantor sets, for example \cite{bernardes-darji,danilenko,downarowicz-maass,fernandez-good-puljiz, shimomura-special, shimomura-scrambled, shimomura-0dim}.

\medskip

The paper is arranged as follows.  In Section 2, we formally define shadowing, shift of finite type and the inverse limit of a direct set of dynamical systems.  In Section 3, we  characterize shadowing as a topological, rather than metric property, and prove that an inverse limit of systems with shadowing itself has shadowing (Theorem \ref{InverseLimitsPreserveShadowing}). Here we also introduce the orbit and pseudo-orbit shift spaces associated with a finite open cover of a dynamical system and observe in Theorem \ref{CoversEncode} that these capture the dynamics of $f$. Section 4 discusses compact, totally disconnected Hausdorff, but not necessarily metric, dynamical systems, showing that such systems have shadowing of and only if they are (conjugate to) the inverse limit of a directed set of shifts of finite type (Theorem \ref{TotallyDiscClass}). In Section 5, we examine the case of general metric systems, establishing in Theorem \ref{GenNecessary} a partial analogue to Theorem \ref{TotallyDiscClass}. In Section 6, we discuss factor maps which preserve shadowing and, in light of this, we are able to completely characterize compact metric systems with shadowing in Theorem \ref{GenClass}.

%It is known \cite{bernardes-darji} that shadowing is generic for homeomorphisms of the Cantor set and \cite{shimomura} that the shifts of finite type form a dense subset of the space of homeomorphisms on the Cantor set. Hirsch \cite{hirsch-expanding} shows that expanding differentiable maps on closed manifolds are factors of the full one sided shift.

%Questions: Does the semiconjugacy in the non zero dim part have the proeprty taht every pseudo orbit has a pull back that is a pseudo orbit?  Can we say anything about minimal subsets in the non-zero-dim case? (see Bowen Markov partitions and minimal sets for Axiom A diffeomorphisms, a nice theorem would be that shadowing plus some generalization of Axiom A implies minimal sets are zero-dim.  If the original $f$ is a homeomorphism can we make the sft bi-infinite sequences? Does this add anything?

%CHECK $(f,X)$ vs $(X,f)$

\section{Preliminaries and Definitions}\label{prelim}

By map, we mean a continuous function. The set of natural numbers (including 0) is denoted by $\omega$.

\begin{dfn} Let $X$ be a compact metric space and let $f:X\to X$ be a  continuous function.
Let $\seq {x_i}$ be a sequence in $X$. Then $\seq {x_i}$ is a \emph{$\delta$-pseudo-orbit} provided $d(x_{i+1},f(x_i))<\delta$ for all $i\in\omega$ and the point $z$ \emph{$\epsilon$-shadows} $\seq {x_i}$ provided $d(x_i,f^i(z))<\epsilon$ for all $i\in\omega$.

The map $f$ has \emph{shadowing} (or the \emph{pseudo-orbit tracing property}) provided that for all $\epsilon>0$ there exists $\delta>0$ such that every $\delta$-pseudo-orbit is $\epsilon$-shadowed by a point.
\end{dfn}

%Much work has been done towards identification and classification of those maps with shadowing. Much of this work stems from Bowen's work on diffeomorphisms on compact manifolds\cite{Bowen}. More recently, a number of authors have shown that shadowing is a relatively common property among maps on certain domains CITATIONS INCLUDING PREPRINT CITATION!.

A particularly nice characterization of shadowing exists if we restrict our attention to \emph{shift spaces}. For a finite set $\Sigma$, the \emph{full one-sided shift with alphabet $\Sigma$} consists of the space of infinite sequence in $\Sigma$, i.e. $\Sigma^\omega$ using the product topology on the discrete space $\Sigma$ and the \emph{shift map} $\sigma$, given by
\[\sigma\seq{x_i}=\seq{x_{i+1}}.\]
A \emph{shift space} is a compact invariant subset $X$ of some full-shift. A shift space $X$ is a \emph{shift of finite type over alphabet $\Sigma$} if there is a finite collection $\mathcal F$ of finite words in $\Sigma$ for which $\seq{x_i}\in\Sigma^\omega$ belongs to $X$ if and only if for all $i\leq j$, the word $x_ix_{i+1}\cdots x_j\notin\mathcal F$. A shift of finite type is said to be \emph{$N$-step} provided that the length of the longest word in its associated set of forbidden words $\mathcal F$ is $N+1$. As mentioned above, a shift space has shadowing if and only if it a shift of finite type \cite{walters}.

Inverse limit constructions arise in a variety of settings. Many of the results here hold for arbitrary (non-metric) compact Hausdorff spaces and so we consider inverse limits of dynamical systems taken along an arbitrary directed set. The reader will not miss much by assuming that the space is compact metric in which case the inverse  limit may be indexed by $\mathbb N$.

\begin{dfn}
	Let $(\Lambda,\leq)$ be a directed set. For each $\lambda\in\Lambda$, let $X_\lambda$ be a compact Hausdorff space and, for each pair $\lambda\leq\eta$, let $g_\lambda^\eta:X_\eta\to X_\lambda$ be a continuous surjection. Then $(g_\lambda^\eta,X_\lambda)$ is called an \emph{inverse system} provided that
	\begin{enumerate}
		\item $g^\lambda_\lambda$ is the identity map, and
		\item for $\lambda\leq\eta\leq\nu$, $g_\lambda^\nu=g_\lambda^\eta\circ g_\eta^\nu$.
	\end{enumerate}
	The \emph{inverse limit of $(g_\lambda^\eta,X_\lambda)$} is the space
	\[\invlim{g_\lambda^\eta}{X_\lambda}=\{\seq{x_\lambda}\in\Pi X_\lambda: \forall \lambda\leq\eta x_\lambda=g_\lambda^\eta(x_\eta)\}\]
	with topology inherited as a subspace of the product $\Pi X_\lambda$.
\end{dfn}

It is well known \cite{engelking} that the inverse limit of compact Hausdorff spaces is itself compact and Hausdorff. Additionally, the following easily proved fact is often useful. If $U\subseteq \invlim{g_\lambda^\eta}{X_\lambda}$ is open, and $x\in U$, then there exists $\lambda$ and $U_\lambda\subseteq X_\lambda$ open with $x\in \pi^{-1}_\lambda(U_\lambda)\cap\invlim{g_\lambda^\eta}{X_\lambda}\subseteq U$. That is, the collection of sets of the form $\pi^{-1}_\lambda(U_\lambda)\cap\invlim{g_\lambda^\eta}{X_\lambda}$ for $U_\lambda$ open in $X_\lambda$ forms a basis for $\invlim{g_\lambda^\eta}{X_\lambda}$.

Now, suppose that for each $\lambda$ in  the directed set $\Lambda$,  $f_\lambda:X_\lambda\to\lambda$ is a continuous function. If the \emph{bonding maps} $g_\lambda^\eta$ commute with the functions $f_\lambda$, then we can extend this definition to the family  of  dynamical systems $\{\sys{f_\lambda}{X_\lambda}:\lambda\in\Lambda\}$.  Specifically we make the following definition.

\begin{dfn}
	Let $(\Lambda,\leq)$ be a directed set. For each $\lambda\in\Lambda$, let $\sys{f_\lambda}{X_\lambda}$ be a dynamical system on a compact Hausdorff space and, for each pair $\lambda\leq\eta$, let $g_\lambda^\eta:X_\eta\to X_\lambda$ be a continuous surjection. Then $(g_\lambda^\eta,\sys{f_\lambda}{X_\lambda})$ is called an \emph{inverse system} provided that
	\begin{enumerate}
		\item $g^\lambda_\lambda$ is the identity map, and
		\item for $\lambda\leq\eta\leq\nu$, $g_\lambda^\nu=g_\lambda^\eta\circ g_\eta^\nu$, and
		\item for $\lambda\leq\eta$, $f_\lambda\circ g_\lambda^\eta=g_\lambda^\eta\circ f_\eta$ (i.e. that $g_\lambda^\eta$ is a semiconjugacy).
			\end{enumerate}
	The \emph{inverse limit of $(g_\lambda^\eta,\sys{f_\lambda}{X_\lambda})$} is the dynamical system $\sys{(f_\lambda)^*}{\invlim{g_\lambda^\eta}{X_\lambda}}$, where $(f_\lambda)^*$ is the \emph{induced map} given by
	\[(f_\lambda)^*\big(\seq{x_\lambda}\big)=\big(f_\lambda(x_\lambda)\big).\]
\end{dfn}

That this defines a continuous dynamical system is a routine exercise. It is also immediate that if each of the maps $f_\lambda$ is surjective, then the {induced map} $(f_\lambda)^*$ is also surjective.

\section{Shadowing without Metrics}

Shadowing is on first inspection a metric property, and indeed the properties of metrics often play a role in the its investigation and application. However, we note in this section that shadowing can be viewed as a strictly topological property, provided that we restrict our attention to compact metric spaces. Similar observations have been made in \cite{brian-abstract, good-macias}.

\begin{dfn}
  Let $X$ be a space, let $f:X\to X$, and let $\mathcal U$ be a finite open cover of $X$.
  \begin{enumerate}
    \item The sequence $\seq{x_i}_{i\in\omega}$ is a \emph{$\mathcal U$-pseudo-orbit} provided for every $i\in\omega$, there exists $U_{i+1}\in\mathcal U$ with  $x_{i+1},f(x_{i})\in U_{i+1}$.  In this case, we say that the sequence $\seq{U_i}$ is a \emph{$\mathcal U$-pseudo-orbit pattern}.
    \item The point $z\in X$ $\mathcal U$-shadows $\seq{x_i}_{i\in\omega}$ provided for each $i\in\omega$ there exists $U_i\in\mathcal U$ with  $x_i,f^i(z)\in U_i$. We call the sequence $\seq {U_i}$ an \emph{$\mathcal U$-orbit pattern}.
  \end{enumerate}
\end{dfn}

\begin{lemma}\label{CompactShadowing}
	Let $X$ be a compact metric space. Then $f:X\to X$ has shadowing if and only if for every finite open cover $\mathcal U$, there exists a finite open cover $\mathcal V$, such that every $\mathcal V$-pseudo-orbit is $\mathcal U$-shadowed by some point $z\in X$. \end{lemma}

\begin{proof}
	First, suppose that $f$ has the shadowing property and let $\mathcal U$ be a finite open cover of $X$. Fix $\epsilon>0$ so that for each $\epsilon$-ball $B$ in $X$, there exists $U\in\mathcal U$ with $\overline B\subseteq U$. Now, let $\delta>0$ witness $\epsilon$-shadowing. Let $\mathcal V$ be a finite open cover of $X$ refining $\mathcal U$ which consists of open sets of diameter less that $\delta$.
	
	Now, Let $\seq{x_i}$ be a sequence in $X$ as in the statement of the lemma. Then $d(x_i,f(x_{i-1}))<diam(V_i)<\delta$ for all $i\in\omega\setminus{0}$. In particular, $\seq{x_i}$ is a $\delta$-pseudo-orbit. Let $z\in X$ be an $\epsilon$-shadowing point for this sequence. Then $d(x_i,f^i(z))<\epsilon$ for each $i\in\omega$, and in particular, $\{x_i,f^i(x)\}\subseteq B_{\epsilon}(x_i)$. By construction, there exists $U_i\in\mathcal U$ for which $\{x_i,f^i(x)\}\subseteq B_{\epsilon}(x_i)\subseteq U_i$, satisfying the conclusion of the lemma.
	
	Conversely, let us suppose that $f$ satisfies the open cover condition of the lemma. Let $\epsilon>0$, and consider a finite subcover $\mathcal U$ of $X$ consisting of $\epsilon/2$-balls. Let $\mathcal V$ be the cover that witnesses the satisfaction of the condition, and choose $\delta>0$ such that for each $\delta$-ball in $X$, there is an element of $\mathcal V$ which contains it.
	
	Now, fix a $\delta$-pseudo-orbit $\seq{x_i}$. Then for each $i\in\omega\setminus{0}$, $d(x_i,f(x_{i-1}))<\delta$, and hence there exists $V_i\in\mathcal V$ such that $x_i,f(x_{i-1})\in V_i$. Let $z\in X$ be the point guaranteed by the open cover condition. Then, for each $i\in\omega$, there exists $U_i\in\mathcal U$ with $x_i,f^i(z)\in U_i$. But $U_i$ is an $\epsilon/2$-ball and hence $d(x_i,f^i(x))<\epsilon$, i.e. $z$ $\epsilon$-shadows the pseudo-orbit.
\end{proof}

This observation allows the decoupling of shadowing from the metric, and we can then take the following definition of shadowing, which is valid for maps on any compact topological space and in particular for systems with compact Hausdorff domain, an application that has recently seen increased interest \cite{Brian-Ramsey,good-macias}.

\begin{dfn} Let $X$ be a (nonempty) topological space. The
	map $f:X\to X$ has \emph{shadowing} provided that for every finite open cover $\mathcal U$, there exists a finite open cover $\mathcal V$ such that
every $\mathcal V$-pseudo-orbit is $\mathcal U$-shadowed by a point of $X$.
\end{dfn}

With this definition in mind, we can prove the following result which will be important to the characterization of shadowing in Section \ref{TotDisc}.

\begin{theorem}\label{InverseLimitsPreserveShadowing}
	Let $f:X\to X$ be conjugate to an inverse limit of maps with shadowing on compact spaces. Then $f$ has shadowing.
\end{theorem}

\begin{proof}
	Without loss, let $(\Lambda,\leq)$ be a directed set and  $\sys{f}{X}=\invlim {g^\lambda_\gamma}{(f_\lambda,X_\lambda)}$ where each of $(f_\lambda,X_\lambda)$ is a system with shadowing on a compact space.
	
	Let $\mathcal U$ be a finite open cover of $X$. Since $X=\invlim{g^\lambda_\gamma}{X_\lambda}$, we can find $\lambda$ and a finite open cover $\mathcal W_\lambda$ of $X_\lambda$ so that $\mathcal W'=\{\pi_\lambda^{-1}(W)\cap X:W\in\mathcal W_\lambda\}$ refines $\mathcal U$. Since $f_\lambda$ has shadowing, we can find a finite open cover $\mathcal V_\lambda$ of $X_\lambda$ witnessing this. Let $\mathcal V=\{\pi_\lambda^{-1}(V)\cap X:V\in\mathcal V_\lambda\}$.
	
	Now, let $\seq{x_i}$ be a $\mathcal V$-pseudo-orbit with pattern $\seq {\pi_\lambda^{-1}(V_i)\cap X}$. Then for each $i\in\omega$, we have $f(\overline {\pi_\lambda^{-1}(V_i)\cap X})\cap\overline {\pi_\lambda^{-1}(V_{i+1})\cap X}\neq\ns$. It follows then, that $\seq{(x_i)_\lambda }$ is a $\mathcal V_\lambda$-pseudo-orbit with pattern $\seq{V_i}$. Since $f_\lambda$ has shadowing, it follows then that there is a sequence $\seq{W_i}$ in $\mathcal W$ with $(x_i)_\lambda\in W_i$ and with $\bigcap f_\lambda^{-i}(\overline W_i)\neq\ns$. It then follows that $x_i\in\pi_\lambda^{-1}(W_i)$ and $\bigcap f^{-i}(\overline {\pi_\lambda^{-1}(W_i)})\neq\ns$, i.e. the pseudo-orbit $\seq{x_i}$ is $\mathcal W'$-shadowed, and hence $\mathcal U$-shadowed as required.
\end{proof}

Now, the dynamics of the map $f$ on $X$ induce dynamics on the \emph{space} of (pseudo-) orbit patterns. This is especially transparent if we take reasonably nice covers. In particular, we are restrict our attention to  \emph{taut} covers, those for which the closure of elements meet only if the elements themselves meet.

\begin{dfn}
	Let $f:X\to X$ be a map on the space $X$, let $\mathcal U$ be a taut open cover and let $\mathcal U^\omega$ be the one-sided shift space on the alphabet $\mathcal U$ with shift map $\sigma$. 	
	\begin{enumerate}
		\item The \emph{$\mathcal U$-orbit space} is the set $\mathcal O(\mathcal U)\subseteq \mathcal U^\omega$ consisting of all sequences $\seq{U_i}$ in $\mathcal U$ for which there exists $z\in X$ with $f^i(z)\in\overline{U_i}$,
		\item The \emph{$\mathcal U$-pseudo-orbit space} is the set $\mathcal {PO}(\mathcal U)\subseteq \mathcal U^\omega$ consisting of all sequences $\seq{U_i}$ in $\mathcal U$ for which there exists a sequence $\seq{x_i}$ with $x_{i+1},f(x_i)\in \overline U_{i+1}$.
	\end{enumerate}
	Additionally, for $U\in\mathcal U$ and $i\in\omega$, define $\pi_i:\mathcal U^\omega\to\mathcal U$ to be projection onto the $i$-th coordinate.
\end{dfn}

The following lemma is immediate and provides an alternate description of $\mathcal O(\mathcal U)$ and $\mathcal {PO}(\mathcal U)$.

\begin{lemma}\label{AltDesc}
	Let $f:X\to X$ be a map on $X$ and let $\mathcal U$ be a taut open cover. Then
	\[\mathcal O(\mathcal U)=\{\seq{U_i}\in \mathcal U^\omega: \bigcap f^{-i}(\overline U_i)\neq\ns \}\] and\[ \mathcal PO(\mathcal U)=\{\seq{U_i}\in\mathcal U^\omega: f(\overline U_{i})\cap \overline U_{i+1}\neq \ns\}.\]
\end{lemma}

As consequence, we have the following relations between $\mathcal O(\mathcal U)$, $\mathcal {PO}(\mathcal U)$ and $\mathcal U^\omega$.

\begin{lemma}\label{FirstOrbitSpaceRelations}
	Let $f:X\to X$ be a map on $X$ and let $\mathcal U$ be a taut open cover. Then, $\mathcal O(\mathcal U)$ is a subset of $\mathcal{PO}(\mathcal U)$ and both spaces are subshifts of $\mathcal U^\omega$. In particular, $\mathcal{PO}(\mathcal U)$ is a 1-step shift of finite type.
\end{lemma}

\begin{proof}
	That $\mathcal O(\mathcal U)\subseteq\mathcal{PO}(\mathcal U)\subseteq \mathcal U^\omega$ is immediate. It is also clear that all of these are shift invariant and closed, and hence are subshifts. That $\mathcal{PO}(\mathcal{U})$ is a 1-step shift of finite follows by observing that the condition that $f(\overline U_{i})\cap \overline U_{i+1}\neq \ns$ is equivalent to forbidding words of the form $UV$ where $f(\overline U)\cap\overline V=\ns$.
\end{proof}

If $X$ is compact Hausdorff, then the entire dynamics of a map $f$ are encoded in the orbit spaces of an appropriate system of covers of $X$. In particular, let $\toc(X)$ be the collection of all finite taut open covers of $X$. This collection is naturally partially ordered by refinement and forms a directed set.  %Recall that a cover $\mathcal V$ (strongly) refines $\mathcal U$ provided that for all $V\in\mathcal V$ there exists $U\in\mathcal U$ with (the closure of) $V$ a subset of $U$. Additionally, if $(\Lambda,\leq)$ is a directed poset we say that a set of open covers $\{\mathcal U_\lambda\}_{\lambda\in\Lambda}$ is a \emph{(strongly) refining system} provided that if $\lambda<\lambda'$ then $\mathcal U_{\lambda'}$ (strongly) refines $\mathcal U_\lambda$.

\begin{theorem}\label{CoversEncode}
	Let $f:X\to X$ be a map on the compact Hausdorff space $X$. Let $\{\mathcal U_\lambda\}_{\lambda\in\Lambda}$ be a cofinal directed subset of $\toc(X)$. Then for all $x\in X$ there exists a choice of $U_\lambda(x)\in\mathcal U_\lambda$ with $\{x\}=\bigcap \overline {U_\lambda(x)}$ and furthermore for any such sequence, we have
%	\[f(x)=\bigcap \pi_0(\sigma(\mathcal {PO}(\mathcal U_\lambda)\cap\pi_0^{-1}(\overline U_\lambda(x)))).\]
%	 and
		for all $n\in\omega$,
	\[\{f^n(x)\}=\bigcap \pi_0(\sigma^n(\mathcal O(\mathcal U_\lambda)\cap\pi_0^{-1}(\overline U_\lambda(x)))).\]
\end{theorem}

\begin{proof}
	Let $f$, $\Lambda$, and $\{\mathcal U_\lambda\}$ be as described. Fix $x\in X$. For each $\lambda\in\Lambda$, choose $U_\lambda(x)\in\mathcal U_\lambda$ with $x\in U_\lambda(x)$. Then $x\in\bigcap\overline{U_\lambda(x)}$. Furthermore, for all $y\in X\setminus\{x\}$, there exists a cover $\mathcal U$ of $X$ such that if $x\in U\in\mathcal U$, then $y\notin\overline U$. Then for any $\lambda\in\Lambda$ with $\mathcal U_\lambda$ refining $\mathcal U$, $y\notin\overline{U_\lambda(x)}$ regardless of choice of $U_\lambda(x)$, and hence $y\notin  \bigcap\overline{U_\lambda(x)}$.
	
	Now, for each $\lambda$, it is straightforward to show that $\pi_0(\sigma^n(\mathcal O(\mathcal U_\lambda)\cap\pi_0^{-1}(\overline {U_\lambda(x)})))$ is equal to $f^n(\overline {U_\lambda(x)})$. In particular, $f^n(x)\in \bigcap \pi_0(\sigma^n(\mathcal O(\mathcal U_\lambda)\cap\pi_0^{-1}(\overline {U_\lambda(x)})))$. Suppose now that $z\in \bigcap \pi_0(\sigma^n(\mathcal O(\mathcal U_\lambda)\cap\pi_0^{-1}(\overline {U_\lambda(x)})))$. Then for each $\lambda$, there exists $x_\lambda\in U_\lambda(x)$ with $z=f^n(x_\lambda)$. But, by construction, $x$ is a limit point of $\{x_\lambda\}$, and by continuity, $z=f^n(x)$. Hence $\{f^n(x)\}=\bigcap \pi_0(\sigma^n(\mathcal O(\mathcal U_\lambda)\cap\pi_0^{-1}(\overline {U_\lambda(x)})))$ as claimed.
\end{proof}

It should be noted that for general Hausdorff spaces, the structure of $\{\mathcal U_\lambda\}$ may be quite complex, but for metric $X$, it is the case that a sequence of covers will always suffice and we will make use of this fact in the following sections.  In the metric case, Theorem \ref{CoversEncode} is equivalent to Theorem 3.9 of \cite{shimomura-special}, although that result is expressed in terms of graph covers and relations.

\section{Characterizing shadowing in totally disconnected spaces}\label{TotDisc}

In the sense of Theorem \ref{CoversEncode}, the entire dynamics are encoded by the action of $f$ on an appropriate collection of refining covers.  This is not unlike the way that the topology is completely encoded as well. In this section we explore this analogy.

In particular, it is well known that a space $X$ is \emph{chainable}, i.e. can be encoded with a sequence of refining \emph{chains} (i.e. finite covers with $U_i\cap U_{j}\neq\ns$ if and only if $|i-j|\leq1$) if and only if $X$ can be written as an inverse limit of arcs. In a sense, the arc is the \emph{fundamental} chainable object. In an analogous fashion we show that shifts of finite type are the fundamental objects among dynamical systems on totally disconnected spaces with shadowing.

For a dynamical system $\sys{f}{X}$ with $X$ totally disconnected in addition to compact Hausdorff, then the covers $\mathcal U_\lambda$ in Theorem \ref{CoversEncode} can be taken to consist of clopen sets which are pairwise disjoint. In this case, let $\mathcal U$ and  $\mathcal V$ be finite clopen pairwise disjoint covers of $X$ with $\mathcal V$ refining $\mathcal U$. Then let $\iota:\mathcal V\to\mathcal U$ be defined so that $V\cap \iota(V)\neq\ns$, which in this context is equivalent to $V\subseteq\iota(V)$. Critically, for $\iota$ to be a function, the covers must consist of pairwise disjoint sets; it is this that creates the obstacle to dealing with non-totally disconnected spaces which  we address in Section \ref{Hausdorff}.
The map $\iota$ naturally induces a continuous map $\iota:\mathcal V^\omega\to \mathcal U^\omega$, the domain of which can then be restricted to $\mathcal O(\mathcal V)$ or $\mathcal{PO}(\mathcal V)$ as appropriate. As the intended domain is typically clear, the symbol $\iota$ will be used for all.

\begin{lemma}\label{OrbitSpaceRelations}
	Let $f:X\to X$ be a map on the compact, totally disconnected Hausdorff space $X$ and let $\mathcal U$ and $\mathcal V$ be finite clopen pairwise disjoint covers of $X$ with $\mathcal V$ refining $\mathcal U$. Then $\sigma$ and $\iota$ commute  and the following statements hold:
	\begin{enumerate}
		\item $\iota(\mathcal O(\mathcal V))=\mathcal O(\mathcal U)$.
		\item $\mathcal O(\mathcal U)\subseteq\iota(\mathcal {PO}(\mathcal V))\subseteq\mathcal {PO}(\mathcal U)$.
%		\item $\iota(\mathcal {PO}(\mathcal V))\supseteq \mathcal O(\mathcal U)$
	\end{enumerate}
\end{lemma}

\begin{proof}
	It is immediate from their definitions that $\sigma$ and $\iota$ commute on their unrestricted domains.
	
	Towards proving statement (1), consider $\seq {V_i}\in\mathcal O(\mathcal V)$. By Lemma \ref{AltDesc}, $\bigcap\overline V_i\neq\ns$, and it follows that $\bigcap\overline {\iota(V_i)}\neq\ns$, and hence $\seq{\iota(V_i)}\in\mathcal O(\mathcal U)$. Conversely, if $\seq{U_i}\in\mathcal O(\mathcal U)$, then we can choose  $x\in \bigcap\overline U_i\neq\ns$. Now choose $\seq {V_i}$ so that $x\in \overline V_i$ for each $i$. Clearly $\seq{V_i}\in\mathcal O(\mathcal V)$. Since $\mathcal V$ refines $\mathcal U$ and the elements of $\mathcal U$ are pairwise disjoint and clopen, it follows that $V_i\subseteq U_i$, i.e. $\seq{U_i}=\iota\seq{V_i}$.
	
	Statement (2) follows similarly.
	
\end{proof}

The additional structure of totally disconnected spaces allows us to state the following immediate corollary to Theorem \ref{CoversEncode}.  In particular, the collection $\pdoc(X)$ of finite pairwise disjoint covers is a cofinal directed subset of $\toc(X)$.

\begin{corollary}\label{TotDiscCoversEncode}
	Let $f:X\to X$ be a map on the compact Hausdorff totally disconnected space $X$. Let $\{\mathcal U_\lambda\}_{\lambda\in\Lambda}$ be a cofinal directed suborder of $\pdoc(X)$.
	
	Then the system $\sys{f}{X}$ is conjugate to $\sys{\sigma^*}{\invlim{\iota}{\mathcal O(\mathcal U_\lambda)}}$ %(where $\sigma^*$ is the map on the inverse limit induced by the shift map $\sigma$ on each factor space)
	by the map
	\[\seq{w_\lambda}\mapsto \bigcap\overline{\pi_0(w_\lambda)}. \]
\end{corollary}

It is important to note that the maps $\iota$ in the inverse system depend very much on their domain and range. However,  if $\mathcal W$ refines $\mathcal V$ which in turn refines $\mathcal U$, then the composition of $\iota:\mathcal W\to\mathcal V$ and $\iota:\mathcal V\to\mathcal U$ is precisely the same as $\iota:\mathcal W\to\mathcal U$, and as such the inverse system is indeed well-defined.

The existence of pairwise disjoint covers also allows us to state the following alternative characterization of shadowing.

\begin{lemma}
	Let $f:X\to X$ be a map on the compact Hausdorff totally disconnected space $X$. Then $f$ has shadowing if and only if for each $\mathcal U\in\pdoc(X)$, there exists $\mathcal V\in\pdoc(X)$ which refines $\mathcal U$ such that for all $\mathcal W\in\pdoc(X)$ which refine $\mathcal V$, $\iota(\mathcal {PO}(\mathcal W))= \mathcal O(\mathcal U)$.
\end{lemma}

\begin{proof}
	Let $f$ have shadowing and let $\mathcal U\in\pdoc(X)$. Let $\mathcal V$ be the cover witnessing shadowing. Without loss of generality, $\mathcal V\in\pdoc(X)$. Now, let $\seq{x_i}$ be a $\mathcal{V}$-pseudo-orbit with  $\seq{V_i}$ its $\mathcal V$-pseudo-orbit pattern, and let $z\in X$ be a shadowing point with $\seq{U_i}$ its shadowing pattern. By definition, $x_i$ and $f^i(z)$ belong to $U_i$, and hence $V_i\cap U_i\neq\ns$, and since $\mathcal V$ refines $\mathcal U$ and the elements of $\mathcal U$ are disjoint, we have $V_i\subseteq U_i$, i.e. $\iota(V_i)=U_i$ and hence $\iota\seq{V_i}=\seq{U_i}$. Thus $\iota(\mathcal{PO}(\mathcal V))\subseteq\mathcal O(\mathcal U)$, and the reverse inclusion is given by Lemma \ref{OrbitSpaceRelations}, and thus the two sets are equal. Now, for any $\mathcal W\in\pdoc(X)$ which refines $\mathcal V$, observe that $\mathcal{O}(\mathcal U)\subseteq \iota(\mathcal {PO}(\mathcal W))\subseteq \iota(\mathcal {PO}(\mathcal V))=\mathcal{O}(\mathcal U)$,
	
	Conversely, suppose that $f$ has the stated property regarding open covers. Let $\mathcal U$ be a finite open cover of $X$. Since $X$ is totally disconnected, let $\mathcal U'\in\pdoc(X)$ which refines $\mathcal U$. Let $\mathcal V$ be the cover witnessing the property with respect to $\mathcal U'$. Now, let $\seq{x_i}$ be a $\mathcal V$-pseudo-orbit and let $\seq{V_i}$ be its $\mathcal V$-pseudo-orbit pattern. By the property, there exists $\seq{U'_i}\in\mathcal O(\mathcal U')$ with $\iota\seq{V_i}=\seq{U'_i}$. Now, let $z\in\bigcap f^{-i}(\overline U'_i)$. Then $x_i,f^i(z)\in\overline U'_i$ which in turn is a subset of some $U_i\in\mathcal U$. In particular, $z$ $\mathcal U$-shadows the $\mathcal V$-pseudo-orbit $\seq{x_i}$.
\end{proof}

In light of this, we have the following theorem.

\begin{theorem}\label{Conjugacy}
	Let $f:X\to X$ be a map with shadowing on the compact totally disconnected Hausdorff space $X$. Let $\{\mathcal U_\lambda\}_{\lambda\in\Lambda}$ be a cofinal directed suborder of $\pdoc(X)$.
	
	Then the system $\sys{\sigma^*}{\invlim{\iota}{\mathcal O(\mathcal U_\lambda)}}$ is conjugate to $\sys{\sigma^*}{\invlim{\iota}{\mathcal {PO}(\mathcal U_\lambda)}}$.
\end{theorem}

\begin{proof}
For each $\lambda\in\Lambda$, let $p(\lambda)\geq\lambda$ such that $\iota(\mathcal {PO}(\mathcal U_{p(\lambda)}))=\mathcal O(\mathcal U_\lambda)$. Then, define the map $\phi: \invlim{\iota}{\mathcal {PO}(\mathcal U_\lambda)})\to \invlim{\iota}{\mathcal O(\mathcal U_\lambda)}$ as follows
	\[\phi(\seq {w_\gamma})_\lambda=\iota(w_{p(\lambda)})\]

That this is well-defined and continuous is a standard result in inverse limit theory. That it is a surjection follows from the surjectivity of $\iota$ from $\mathcal{PO}(\mathcal U_{p(\lambda)})$ onto $\mathcal O(\mathcal U_\lambda)$. As this is induced by the maps $\iota$, it will commute with $\sigma^*$. All that remains is to demonstrate injectivity.

To this end, suppose that $\seq{w_\gamma}\neq\seq{v_\gamma}$ but $\phi(\seq{w_\gamma})=\phi(\seq{v_\gamma})$.  Fix $\lambda\in\Lambda$ with $w_\lambda\neq v_\lambda$. Since these are elements of the inverse limit and the elements of $\mathcal U_\lambda$ are disjoint, for all $\gamma\geq\lambda$, we have $w_\gamma\neq v_\gamma$. In particular, then, there is an $n\in\omega$ with $(w_\lambda)_n\neq(v_\lambda)_n$, and hence for all $\gamma\geq\lambda$, $(w_\gamma)_n\neq(v_\gamma)_n$. Thus the nested intersections of these are nonempty and distinct, i.e.
\[\ns\neq\bigcap_{\gamma\geq\lambda}\overline{(w_\gamma)_n}\neq\bigcap_{\gamma\geq\lambda}\overline{(v_\gamma)_n}\neq\ns.\]

This intersection may be taken over fewer sets, in particular we may restrict to those indices $\gamma$ which are equal to $p(\eta)$ for some $\eta$, so we have
\[\ns\neq\bigcap_{\gamma\geq\lambda}\overline{(w_{p(\gamma)})_n}\neq\bigcap_{\gamma\geq\lambda}\overline{(v_{p(\gamma)})_n}\neq\ns.\]

On the other hand, since $\phi(\seq{w_\gamma})=\phi(\seq{v_\gamma})$, we have
\[\bigcap_{\gamma\geq\lambda}\overline{\iota(w_{p(\gamma)})_n}=\bigcap_{\gamma\geq\lambda}\overline{\iota(v_{p(\gamma)})_n}.\]
However all of these intersections are singletons and since $\iota(w_{p(\gamma)})_n\supseteq (w_{p(\gamma)})_n$, we have
\[\bigcap_{\gamma\geq\lambda}\overline{(w_{p(\gamma)})_n}=\bigcap_{\gamma\geq\lambda}\overline{\iota(w_{p(\gamma)})_n}=\bigcap_{\gamma\geq\lambda}\overline{\iota(v_{p(\gamma)})_n}=\bigcap_{\gamma\geq\lambda}\overline{(v_{p(\gamma)})_n},\]
which is a contradiction.

\end{proof}

This theorem complements Corollary \ref{TotDiscCoversEncode}, and by applying Lemma \ref{FirstOrbitSpaceRelations}, and the well-known fact that shifts of finite type have shadowing \cite{walters}, we have the following result.

\begin{corollary}\label{Sufficiency}
	Let $f:X\to X$ be a map with shadowing on the compact totally disconnected Hausdorff space $X$. Then $\sys{f}{X}$ is conjugate to an inverse limit of shifts of finite type.
\end{corollary}

In fact, this is a complete characterization of totally disconnected systems with shadowing; the following is an immediate consequence of Corollary \ref{Sufficiency} and Theorem \ref{InverseLimitsPreserveShadowing}.

\begin{theorem}\label{TotallyDiscClass}
	Let $X$ be a compact,  totally disconnected Hausdorff space.  The map $f:X\to X$ has shadowing if and only if $\sys{f}{X}$ is conjugate to an inverse limit of shifts of finite type.
\end{theorem}

Of course, Theorem \ref{TotallyDiscClass} includes metric systems. However, if $X$ is metric, we may easily find sequences $\{\mathcal U_n\}_{n\in\N}$ of finite pairwise disjoint open covers which are cofinal directed suborders of $\pdoc(X)$. In particular, we can let $\mathcal U_0=\{X\}$, and for each $\mathcal U_i$, let $\mathcal U_{i+1}$ be a pairwise disjoint finite open cover of $X$ with mesh less than $2^{-1}$ which refines $\mathcal U_i$ and which witnesses shadowing for $\mathcal U_i$. Then the function $p$ from the proof of Theorem \ref{Conjugacy} simply increments its input. The conjugacy then follows from the induced diagonal map $\iota^*$ on the inverse systems as seen in Figure \ref{DiagonalMap}.

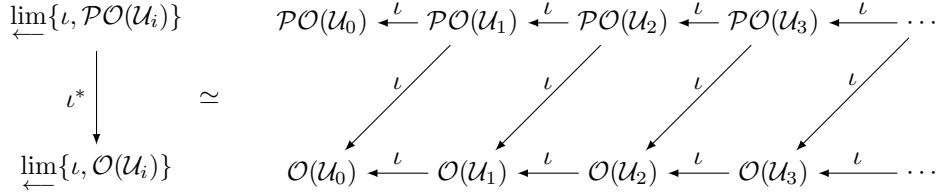
\begin{figure}[h]
	\begin{center}
		\input{DiagonalMap.tex}
	\end{center}
	\caption{Diagram for the metric case of Theorem \ref{TotallyDiscClass}}\label{DiagonalMap}
\end{figure}

This observation immediately implies the following.

\begin{theorem}\label{cantor-shadowing} Let $X$ be the Cantor set, or indeed any compact, totally disconnected metric space.
  The map $f:X\to X$  has shadowing if and only if $\sys{f}{X}$ is conjugate to the inverse limit of a sequence of shifts of finite type.
\end{theorem}

This ad hoc construction of an appropriate sequence of covers can be modified into a technique that will apply to general compact metric spaces in Section \ref{Hausdorff}.

%\begin{corollary}\label{TotallyDiscMetric}
%	Let $f:X\to X$ be a map on the compact totally disconnected metric  space $X$. Then $f$ has shadowing if and only if there is an inverse \emph{sequence} $(\iota,(\sigma,X_n))$ of shifts of finite type $X_n$ with $(f,X)$ conjugate to $(\sigma^*,\invlim{\iota}{(\sigma,X_n)})$.
%\end{corollary}
%
%\begin{proof}
%	By Lemma \ref{InverseLimitsPreserveShadowing}, we need only demonstrate that for each system $(f,X)$ with shadowing on a compact totally disconnected metric space there exists such a system. Towards this end, let
%	
%	It then follows
%\end{proof}

We complete the section with another characterization of dynamical systems with shadowing on totally disconnected compact metric spaces. Once again, this result is analogous to a characterization of chainable continua in terms of $\epsilon$-maps. Recall that an \emph{$\epsilon$-map} is a map $\phi:X\to Y$ such that for all $y\in Y$, $\phi^{-1}(y)$ has diameter less than $\epsilon$. It is a routine exercise to prove that if $\phi:X\to Y$ is an $\epsilon$-map and $Y$ is compact metric, then for each $c>0$, there exists $\eta>0$ such that for all $U\subseteq Y$ with diameter less than $\eta$, $\phi^{-1}(U)$ has diameter less than $\epsilon+c$.

\begin{theorem}
	Let $f:X\to X$ be a map on the totally disconnected compact metric space $X$. Then $f$ has shadowing if and only if for all $\epsilon>0$, there exists a shift of finite type $X_\epsilon$ and an $\epsilon$-map $\phi:X\to X_\epsilon$ such that $\phi$ is a semiconjugacy.
\end{theorem}

\begin{proof}
	First, observe that if $f$ has shadowing, then it is conjugate to an inverse limit of a \emph{sequence} of shifts of finite type as per the discussion following Theorem \ref{TotallyDiscClass}. In this case, with the usual metric on the product, the projection map onto the $n$-th factor space is a $\frac{1}{2^n}$-map and is also a semiconjugacy as required.
	
	Now, suppose that for each $\gamma>0$, $f:X\to X$ is semiconjugate by a $\gamma$-map to a shift of finite type $X_\gamma$. Fix a particular $\epsilon>0$ and let $\phi:X\to X_{\epsilon/2}$ be an $\epsilon/2$-map and semiconjugacy. Now, find $\eta>0$ so that if $U\subseteq X_{\epsilon/2}$ has diameter less than $\eta$, then the diameter of $\phi^{-1}(U)$ is less than $\epsilon$. Since $X_{\epsilon/2}$ is a shift of finite type, the shift map has shadowing on $X_{\epsilon/2}$, so let us choose  $\delta'$ so that every $\delta$-pseudo-orbit is  $\eta$-shadowed. Finally, by uniform continuity of $\phi$ choose $\delta>0$ so that if $d(x,y)<\delta$ that $d(\phi(x),\phi(y))<\delta'$.
	
	Now, let $\seq{x_i}$ be a $\delta$-pseudo-orbit for $f$. Then $\seq{\phi(x_i)}$ is a $\delta'$-pseudo-orbit for $\sigma$. As such, there exists $z'\in X_{\epsilon/2}$ which $\eta$-shadows it, in particular, the diameter of $\{\sigma^i(z'),\phi(x_i)\}$ is less than $\eta$, and so the diameter of $\phi^{-1}(\sigma^i(z'))\cup \{x_i\}$ is less than $\epsilon$. Now, let $z\in\phi^{-1}(z')$. Then $f^i(z)\in \phi^{-1}(\sigma^i(z))$, and so $d(f^i(z),x_i)<\epsilon$, i.e. $z$ $\epsilon$-shadows $\seq{x_i}$.
\end{proof}

\section{Shadowing in Systems with Connected Components}\label{Hausdorff}

Theorem \ref{CoversEncode} applies equally well to systems in which there are non-trivial connected components, and as such, one might hope for analogue to Corollary \ref{TotDiscCoversEncode}.
%
%\begin{corollary}\label{GenCoversEncode}
%		Let $f:X\to X$ be a map on the compact Hausdorff space $X$. Let $\{\mathcal U_\lambda\}_{\lambda\in\Lambda}$ be a cofinal directed suborder of $\toc(X)$.
%		
%		Then the system $(f,X)$ is semiconjugate to $(\sigma^*,\invlim{\iota}{\mathcal O(\mathcal U_\lambda)})$ %(where $\sigma^*$ is the map on the inverse limit induced by the shift map $\sigma$ on each factor space)
%		by the map
%		\[\seq{w_\lambda}\mapsto \bigcap\overline{\pi_0(w_\lambda)}. \]
%\end{corollary}

However, as mentioned, the principal obstruction to a direct application of the methods of Section \ref{TotDisc} is that the intersection relation $\iota$ is no longer necessarily single-valued, so that the induced map on the pseudo-orbit space is not only set-valued, but also not finitely determined. However, by modifiying the approach illustrated in Figure \ref{DiagonalMap}, we are obtain the following.

\begin{theorem}\label{GenNecessary} Let $X$ be a compact Hausdorff space and $f:X\to X$ be continuous. 
	Suppose that $f$ has shadowing and that $\seq{\mathcal {U}_i}$ is a sequence of finite open covers satisfying the following properties:
	\begin{enumerate}
		\item $\mathcal U_{n+1}$ witnesses $\mathcal U_n$ shadowing,
		\item $\{\mathcal U_{i}\}$ is cofinal in $\toc(X)$, and
		\item for all $U\in\mathcal U_{n+2}$, there exists $W\in\mathcal U_n$ such that $st(U,\mathcal U_{n+1})\subseteq W$.
	\end{enumerate}
	Then there is an inverse sequence $(g_n^{n+1}, X_n)$ of shifts of finite type such that  $\sys{f}{X}$ is semiconjugate to $\sys{\sigma^*}{\invlim{g_n^{n+1}}{X_n}}$.
\end{theorem}

\begin{proof}
	Let $f:X\to X$ and covers $\seq {\mathcal U_i}$ be as stated. For each $U\in\mathcal U_{n+2}$, fix $W(U)\in\mathcal U_n$ with $st(U,\mathcal U_{n+1})\subseteq W(U)$, and define $w:\mathcal{PO}(\mathcal U_{n+2})\to\prod\mathcal U_n$ by $w(\seq{U_j})=\seq{W(U_j)}$. Note that, as this is a single letter substitution map on a shift space, it is a continuous map and commutes with the shift map by definition.
	
	We claim that $w(\mathcal{PO}(\mathcal U_{n+2}))$ is a subset of $\mathcal{O}(\mathcal U_{n})$. Indeed, let $\seq{U_j}\in\mathcal{PO}(\mathcal U_{n+2})$ and $\seq{x_j}$ a pseudo-orbit with this pattern. Since $\mathcal U_{n+2}$ witnesses $\mathcal U_{n+1}$-shadowing, there exists $z\in X$ and sequence $\seq{V_j}\in \mathcal O(\mathcal U_{n+1})$ with $f^j(z),x_j\in V_j$. In particular, for any such $z$ and choices of $\seq{V_j}$,  $V_j\subseteq st(U_j,\mathcal U_{n+1})\subseteq W(U_j)$. Indeed, this establishes that $\seq{W(U_j)}\in\mathcal{O}(\mathcal U_{n})$. It should be noted that while $w$ is not necessarily surjective, for every $x\in X$, there is some $\seq{U_j}$ in $w(\mathcal{PO}(\mathcal U_{n+2}))$ with $f^j(x)\in U_j$ for all $j$. We can observe this by noting that $\seq{f^j(x)}$ is itself a $\mathcal U_{n+2}$-pseudo-orbit, and in particular, we have $f^j(x)\in W(U_j)$, and so $\seq{W(U_j)}$ is a $\mathcal{U_n}$ orbit pattern for $x$.
	
	Since $\mathcal O(\mathcal U_n)\subseteq\mathcal{PO}(\mathcal U_{n})$, we have the following diagram.

	\begin{figure}[h]
		\begin{center}
			\input{DiagonalMap2.tex}
		\end{center}
		\caption{Diagram for the proof of Theorem \ref{GenNecessary}}\label{DiagonalMap2}
	\end{figure}
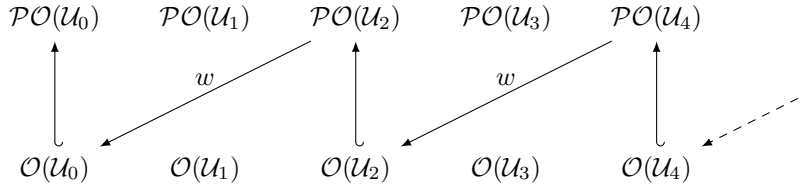
	
	So, while the `natural' map from $\mathcal{PO}({\mathcal U_{n+2}})$ is set-valued, the composition of inclusion and $w$ gives a single-valued continuous map from $\mathcal{PO}({\mathcal U_{n+2}})$ to $\mathcal{PO}({\mathcal U_{n}})$, and by reversing the order of composition, from $\mathcal{PO}({\mathcal U_{n+2}})$ to $\mathcal{PO}({\mathcal U_{n}})$. We will denote these maps by $\iota'$. Figure \ref{DiagonalMap3} then establishes the conjugacy between the inverse limits of the pseudo-orbit spaces and orbit spaces. Once again, while these maps are not surjective, it is important to note that they do preserve much of the structure. In particular, for each $x\in X$, and each $n\in\N$ and $k\in\omega$, there is some element of $w\circ(\iota')^n(\mathcal{PO}(\mathcal U_{2(n+1)+2k}))$ which is a $\mathcal U_{2k}$-orbit pattern for $x$. This follows immediately from the comments preceding Figure \ref{DiagonalMap2}.
	
	\begin{figure}[h]
		\begin{center}
			\input{DiagonalMap3.tex}
		\end{center}
		\caption{Diagram for the proof of Theorem \ref{GenNecessary}}\label{DiagonalMap3}
	\end{figure}
	
	All that remains is to establish that the inverse limit of orbit spaces is semiconjugate to the system $\sys{f}{X}$. Let $\phi:\invlim{\iota'}{\mathcal{O}(\mathcal U_{2i})}\to X$ be given by $\phi\seq{u_i}=\bigcap\overline{\pi_0(u_i)}$. Note that, by construction, $\pi_k(u_{i+1})\subseteq\pi_k(u_i)$ for all $k$ and $i$, so in particular $\phi(\seq{u_i})$ is a nested intersection of the closures of elements of the open covers, and hence is well-defined. That $\phi$ is continuous and commutes with $\sigma^*$ follows from similar reasoning as Theorem \ref{CoversEncode} and Corollary \ref{TotDiscCoversEncode}. All that remains is to establish that $\phi$ is surjective, but this follows from the aforementioned fact that for all $x\in X$, and all $k,n$, there exists a $\mathcal U_{2k}$-orbit pattern for $x$ in $w\circ(\iota')^n(\mathcal{PO}(\mathcal U_{2(n+1)+2k}))$.
	
	\end{proof}
	
	Clearly, the existence of a sequence of covers satisfying conditions (2) and (3) of Theorem \ref{GenNecessary} is a strong condition. In particular, the fact that this sequence is cofinal in $\toc(X)$ implies that the space $X$ is metrizable. In fact, it is the case that for all compact metric spaces, conditions (2) and (3) are easily satisfiable, yielding the following.
	
	\begin{corollary}\label{MetricGen}
		Let $f:X\to X$ be a map with shadowing on the compact metric space $X$. Then there is an inverse sequence $(g_n^{n+1}, X_n)$ of shifts of finite type such that $\sys{f}{X}$ is semiconjugate to $\sys{\sigma^*}{\invlim{g_n^{n+1}}{X_n}}$.
	\end{corollary}
	
	\begin{proof}
		By Theorem \ref{GenNecessary}, we need only demonstrate the existence of a sequence of covers satisfying (1), (2), and (3). This is easily accomplished by taking $\mathcal U_0=\{X\}$, and inductively letting $\mathcal U_{n+1}$ be a cover witnessing $\mathcal U_n$-shadowing with mesh less than one third the Lebesgue number of the cover $\mathcal U_n$. Conditions (1) and (2) are immediately met. To verify that condition (3) is satisfied, fix $n\in\N$ and $U\in\mathcal U_{n+2}$. Then $U$ is a subset of $V$ for some $V\in\mathcal U_{n+1}$, and so $st(U,\mathcal U_{n+1})$ is a subset of $st(V,\mathcal U_{n+1})$. But the diameter of $st(V,\mathcal U_{n+1})$ is at most three times the mesh of $\mathcal U_{n+1}$, and hence has diameter less than the Lebesgue number for $\mathcal U_n$. Hence there is some $W\in\mathcal U_n$ for which $W\supseteq st(V,\mathcal U_{n+1})\supseteq st(V,\mathcal U_{n+1})$ as required.
	\end{proof}
	
\section{Factor maps which preserve shadowing}

We have now established that the existence of a semiconjugate inverse limit of an inverse sequence of shifts of finite type is necessary for a metric system to exhibit shadowing. However, it is worth noting that this is by no means sufficient. In particular, every sofic shift is semiconjugate to such an inverse limit, but only those that are shifts of finite type exhibit shadowing.

\begin{remark}
	Let $X$ be the subshift of $\{0,1\}^\Z$ consisting of those bi-infinite words containing at most one 1. The system $(\sigma,X)$ fails to have shadowing, but is semiconjugate to the inverse limit of an inverse sequence of shifts of finite type.
\end{remark}

\begin{proof}
	Let $Y$ be the subshift of $\{0,1,2\}^\Z$ consisting of those sequences in which the words $02,10,21$ and $20$ do not appear, i.e. $Y$ is the subshift of all sequences of the form $\cdots 0000001222222\cdots$ along with the constant sequences $\seq0$ and $\seq2$. $Y$ is a shift of finite type. Then $Y$ is (trivially) an inverse limit of shifts of finite type. However, the map from $Y$ to $X$ induced by substituting the symbol $0$ for $2$ is a semiconjugacy.
\end{proof}

It is then natural to ask which factors of inverse limits of shifts of finite type exhibit shadowing, i.e. is there a class of maps $\mathcal{P}$ such that if $\sys{f}{X}$ is a factor by a map in $\mathcal P$ of an inverse limit of shifts of finite type, then $\sys{f}{X}$ has shadowing? Of course, it is clear that there is such a class, and in fact the class of homeomorphisms have this property, but we wish to find, if possible, the maximal such class. Towards this end, we define the following.

\begin{dfn}
	Let $\sys{f}{X}$ and $\sys{g}{Y}$ be dynamical systems with $X$ and $Y$ compact Hausdorff spaces. A semiconjugacy $\phi:\sys{f}{X}\to \sys{g}{Y}$ \emph{lifts pseudo-orbits} provided that for every $\mathcal V_X\in\toc(X)$, there exists $\mathcal V_Y\in\toc(Y)$ such that if $\seq{y_i}$ is a $\mathcal V_Y$-pseudo-orbit in $Y$, then there is a $\mathcal V_X$-pseudo-orbit $\seq{x_i}$ in $X$ with $\seq{y_i}=\seq{\phi(x_i)}$.
\end{dfn}

\begin{theorem}\label{lifts-ps-orbs}
	Let $\sys{f}{X}$ and $\sys{g}{Y}$ be dynamical systems with $X$ and $Y$ compact Hausdorf. If $\sys{f}{X}$ has shadowing and $\phi:(f,X)\to(g,Y)$ is a semiconjugacy that lifts pseudo-orbits, then $\sys{g}{Y}$ has shadowing.
\end{theorem}

\begin{proof}
	Fix an open cover $\mathcal U_Y\in\toc(Y)$, and let $\mathcal U_X\in\toc(X)$ such that $\phi(U_X)$ refines $\mathcal U_Y$. Since $\sys{f}{X}$ has shadowing, let $\mathcal V_X\in\toc(X)$ witness shadowing with respect to $\mathcal U_X$.
	
	Since $\phi$ lifts pseudo-orbits, let $\mathcal V_Y$ witness this with respect to $\mathcal V_X$. Finally, let $\seq{y_i}$ be a $\mathcal V_Y$-pseudo-orbit.
	
	Pick $\seq{x_i}$ to be a $\mathcal V_X$-pseudo-orbit with $\seq{\phi(x_i)}=\seq{y_i}$. As every $\mathcal V_X$-pseudo-orbit is $\mathcal U_X$-shadowed, fix $z_X\in X$ to witness this and let $z_Y=\phi(z_X)$. It then follows that for each $i$, we have $\phi^i(z_X),\phi(x_i)\in\phi(U_{X,i})$ for some $U_i\in\mathcal U_X$. As $\phi(\mathcal U_X)$ refines $\mathcal U_Y$, it follows that there exists $U_{Y,i}\in\mathcal U_Y$ with $\phi^i(z_Y),y_i\in\phi(U_{Y,i})$, i.e. $z_Y=\phi(z_X)$ $\mathcal U_Y$-shadows $\seq{y_i}$.
\end{proof}

Notwithstanding Theorem \ref{lifts-ps-orbs}, a more general concept of lifting pseudo-orbits provides a much sharper insight into the relation between shadowing and shifts of finite type in compact metric systems.

\begin{dfn}
	Let $\sys{f}{X}$ and $(g,Y)$ be dynamical systems with $X$ and $Y$ compact Hausdorff spaces. A semiconjugacy $\phi:\sys{f}{X}\to (g,Y)$ \emph{almost lifts pseudo-orbits} (or $f$ \emph{is an ALP map}) provided that for every $\mathcal V_X\in\toc(X)$ and $\mathcal W_Y\in\toc(Y)$, there exists $\mathcal V_Y\in\toc(Y)$ such that if $\seq{y_i}$ is a $\mathcal V_Y$-pseudo-orbit in $Y$, then there is a $\mathcal V_X$-pseudo-orbit $\seq{x_i}$ in $X$ such that for each $i\in\N$ there exists $W_i\in\mathcal W_Y$ with $\phi(x_i),y_i\in W_i$.
\end{dfn}

\begin{theorem} \label{ALP-maps}
	Let $\sys{f}{X}$ and $\sys{g}{Y}$ be dynamical systems with $X$ and $Y$ compact Hausdorf and let $\phi:\sys{f}{X}\to(g,Y)$ be a semiconjugacy. Then the following statements hold:
	\begin{enumerate}
		\item if $\sys{f}{X}$ has shadowing and $\phi$ is an ALP map, then $\sys{g}{Y}$ has shadowing, and
		\item if $\sys{g}{Y}$ has shadowing then $\phi$ is an ALP map.
	\end{enumerate}
\end{theorem}

\begin{proof}
	First, we prove statement (1). Let $\sys{f}{X}$ have shadowing and let $\phi$ be an ALP map. Fix an open cover $\mathcal U_Y\in\toc(Y)$. Let $\mathcal W_Y\in\toc(Y)$ such that if $W,W'\in\mathcal W_Y$ with $W\cap W'\neq\ns$, then there exists $V\in\mathcal V_Y$ with $W\cup W'\subseteq V$, and let $\mathcal U_X\in\toc(X)$ such that $\phi(U_X)$ refines $\mathcal W_Y$. Since $\sys{f}{X}$ has shadowing, let $\mathcal V_X\in\toc(X)$ witness shadowing with respect to $\mathcal U_X$.
	
	Since $\phi$ is ALP, let $\mathcal V_Y$ witness this with respect to $\mathcal W_Y$ and $\mathcal V_X$. Finally, let $\seq{y_i}$ be a $\mathcal V_Y$-pseudo-orbit.
	
	Pick $\seq{x_i}$ to be a $\mathcal V_X$-pseudo-orbit so that $\seq{\phi(x_i)}$ $\mathcal W_Y$-shadows $\seq{y_i}$. As every $\mathcal V_X$-pseudo-orbit is $\mathcal U_X$-shadowed, fix $z_X\in X$ to witness this and let $z_Y=\phi(z_X)$. It then follows that for each $i$, we have $\phi^i(z_X),\phi(x_i)\in\phi(U_{X,i})$ for some $U_i\in\mathcal U_X$. As $\phi(\mathcal U_X)$ refines $\mathcal W_Y$, it follows that there exists $W_i\in\mathcal W_Y$ with $\phi^i(z_Y),\phi(x_i)\in W_i$. Additionally, as pseudo-orbits are almost lifted, there exists $W'_i\in\mathcal W_Y$ with $\phi(x_i), y_i\in W'_i$. In particular, $\phi(x_i)\in W_i\cap W'_i$, so we have that there exists $U_{Y,i}\in\mathcal U_Y$ with $\phi^i(z_Y),\phi(x_i), y_i\in W_i\cup W'_i\subseteq U_{Y,i}\in \mathcal U_Y$, i.e $z_Y$ $\mathcal U_Y$-shadows $\seq{y_i}$.
	
	Now, to prove statement (2), assume that $\sys{g}{Y}$ has shadowing. let $\mathcal V_X\in\toc(X)$ and $\mathcal W_Y\in\toc(Y)$. Let $\mathcal V_Y\in\toc(Y)$ witness shadowing with respect to $\mathcal W_Y$. Now, let $\seq{y_i}$ be a $\mathcal V_Y$ pseudo-orbit in $Y$. Then there is some $z\in Y$ with $z$ $\mathcal W_Y$-shadowing $\seq{y_i}$. Now, choose $x\in\phi^{-1}(z)$ and observe that $\seq{f^i(x)}$ is a $\mathcal V_X$-pseudo-orbit (as it is in fact an orbit), and $\phi(f^i(x))=g^i(z)$, so there exists $W_i\in\mathcal W_Y$ with $\phi(f^i(z)),y_i\in W_i$ (as given by the shadowing pattern of $z$ and $\seq{y_i}$). Thus $\phi$ is an ALP map.
\end{proof}

%Of note in the above proof is that when $\sys{g}{Y}$ has shadowing, we are able to almost lift the pseudo-orbit $\seq{y_i}$ to an \emph{orbit} in $\sys{f}{X}$.

One consequence of the above is that the (semi-)conjugacies in Theorems \ref{TotallyDiscClass} and \ref{GenNecessary} and in Corollary \ref{MetricGen} are ALP maps. This is not terribly surprising in the case of Theorem \ref{TotallyDiscClass}, as the map in question is a homeomorphism, but in the case of Theorem \ref{GenNecessary} and Corollary \ref{MetricGen}, this is interesting, and allows us to complete the classifications. It should be also be noted that, as a result of this theorem, we see that the factor maps constructed by Bowen in \cite{bowen-markov-partitions} are ALP.

To complete the characterization, we first translate the notion of almost lifting pseudo-orbits from the language of covers into the language of metric spaces. This is not completely necessary, but allows for a different perspective on the property. As this is a direct translation, we state the following without proof.

\begin{lemma}
	Let $\sys{f}{X}$ and $(g,Y)$ be dynamical systems with $X$ and $Y$ compact metric spaces. A semiconjugacy $\phi:\sys{f}{X}\to (g,Y)$ is an ALP map if and only if for all $\epsilon>0$ and $\eta>0$, there exists $\delta>0$ such that if $\seq{y_i}$ is a $\delta$-pseudo-orbit in $Y$, there exists an $\eta$-pseudo-orbit $\seq{x_i}$ in $X$ with $d(\phi(x_i),y_i)<\epsilon$.
\end{lemma}

\begin{theorem} \label{GenClass} Let $X$ be a compact metric space. The map $f:X\to X$ has shadowing if and only if it is semiconjugate to the inverse limit of a sequence of shifts of finite type by a map which is ALP.
\end{theorem}

\begin{proof}
	This follows immediately from Corollary \ref{MetricGen} and part (2) of Theorem \ref{ALP-maps}.
\end{proof}

%
%\begin{lemma}
%	Let $\sys{f}{X}$ be a dynamical system with $X$ compact Hausdorff and let $\seq{\mathcal U_i}$ satisfy the hypotheses of Theorem \ref{GenNecessary}. Then the map $\phi:\sys{\sigma^*}{\invlim{\iota'}{\mathcal O(\mathcal U_{2i})}}\to \sys{f}{X}$ given by $\phi(\seq{u_{2i}})=\bigcap\overline{u_{2i}}$ almost lifts pseudo-orbits.
%\end{lemma}
%
%\begin{proof}
%	Fix $\mathcal V\in\toc(\invlim{\iota'}{\mathcal O(\mathcal U_{2i})})$ and $\mathcal W_X\in\toc(X)$. Now, choose an $n\in\N$ such that $\mathcal U_{2n}$ refines $\mathcal W_X$ and $\pi^{-1}_{2n}(\mathcal U_{2n})$ refines $\mathcal V$ and then choose a $\mathcal U_{2(n+2)}$-pseudo-orbit $\seq{y_i}$. Let $\seq{U_i}$ be its pseudo-orbit pattern. By construction, $\seq{U_i}\in\mathcal{PO}(\mathcal U_{2(n+2)})$ and $w(\seq{U_i})\in\mathcal O(\mathcal U_{2n})$. Finally, choose $u_i\in\invlim{\iota'}{\mathcal O(\mathcal U_{2i})}$ with $\pi_{2n}(u_i)=w(U_i)$. Then $\seq {u_i}$ is a $\pi^{-1}_{2n}(\mathcal U_{2n})$-pseudo-orbit and hence a $\mathcal V$ pseudo-orbit. Furthermore, $\phi(u_i), y_i\in w(U_i)$ which is a subset of some $W_i\in \mathcal W_X$, i.e. $\seq{\phi(u_i)}$ $\mathcal W_X$-shadows $\seq{y_i}$ as desired.
%\end{proof}
%
%A similar proof demonstrates that the conjugacy in Theorem \ref{GenNecessary}. These observations along with the results of Theorems \ref{Conjugacy} and \ref{GenNecessary} allow us to complete the characterization of compact metric systems with shadowing.

Of course, it would be of significant benefit if ALP maps had an alternate characterization. In particular, it is clear that homeomorphisms and covering maps lift pseudo-orbits. However, there are maps which are neither covering maps nor homemorphisms which almost lift pseudo-orbits, in particular, the semiconjugacies given in Theorem \ref{GenNecessary} and Corollary \ref{MetricGen} are not typically open, much less covering maps.

\bibliographystyle{plain}
\bibliography{ComprehensiveBib}

\end{document}

%% file: DiagonalMap.tex
\begin{tikzpicture}[line cap=round,line join=round,x=1.0cm,y=1.0cm,scale=.5,>=latex]
%\begin{scriptsize}
%\clip(-3,-3) rectangle (3,3);
%\draw (0,0) ellipse (1 and 1);
%\node[] at (-3,-3) {$X$};
%\draw plot [->, smooth] coordinates {(0,1)(-1,0)(0,-1)(1,0)(0,1)};
%\node[fill=black, circle, inner sep=0, minimum size=3pt] at (0,0) (x0) {$$};

\node at (0,0) (x0) {$\mathcal{O}(\mathcal U_0)$};
\node at (4,0) (x1) {$\mathcal{O}(\mathcal U_1)$};
\node at (8,0) (x2) {$\mathcal{O}(\mathcal U_2)$};
\node at (12,0) (x3) {$\mathcal{O}(\mathcal U_3)$};
\node at (16,0) (x4)  {$\cdots$};
\node at (-6,0) (xinf) {$\invlim{\iota}{\mathcal{O}(\mathcal U_i)}$};
\path (x1) edge[->] node[above]{$\iota$} (x0);
\path (x2) edge[->] node[above]{$\iota$} (x1);
\path (x3) edge[->] node[above]{$\iota$} (x2);
\path (x4) edge[->] node[above]{$\iota$} (x3);
\node at (0,4) (y0) {$\mathcal{PO}(\mathcal U_0)$};
\node at (4,4) (y1) {$\mathcal{PO}(\mathcal U_1)$};
\node at (8,4) (y2) {$\mathcal{PO}(\mathcal U_2)$};
\node at (12,4) (y3) {$\mathcal{PO}(\mathcal U_3)$};
\node at (16,4) (y4)  {$\cdots$};
\node at (-6,4) (yinf) {$\invlim{\iota}{\mathcal{PO}(\mathcal U_i)}$};
\path (y1) edge[->] node[above]{$\iota$} (y0);
\path (y2) edge[->] node[above]{$\iota$} (y1);
\path (y3) edge[->] node[above]{$\iota$} (y2);
\path (y4) edge[->] node[above]{$\iota$} (y3);
\path (y1) edge[->] node[above]{$\iota$} (x0);
\path (y2) edge[->] node[above]{$\iota$} (x1);
\path (y3) edge[->] node[above]{$\iota$} (x2);
\path (y4) edge[->] node[above]{$\iota$} (x3);
\path (yinf) edge[->] node[left]{$\iota^*$} (xinf);
\node at (-3,2) (mid) {$\simeq$};
\end{tikzpicture}

%% file: DiagonalMap2.tex
\begin{tikzpicture}[line cap=round,line join=round,x=1.0cm,y=1.0cm,scale=.5,>=latex]
%\begin{scriptsize}
%\clip(-3,-3) rectangle (3,3);
%\draw (0,0) ellipse (1 and 1);
%\node[] at (-3,-3) {$X$};
%\draw plot [->, smooth] coordinates {(0,1)(-1,0)(0,-1)(1,0)(0,1)};
%\node[fill=black, circle, inner sep=0, minimum size=3pt] at (0,0) (x0) {$$};

\node at (0,0) (x0) {$\mathcal{O}(\mathcal U_0)$};
\node at (4,0) (x9) {$\mathcal{O}(\mathcal U_1)$};
\node at (8,0) (x1) {$\mathcal{O}(\mathcal U_2)$};
\node at (12,0) (x10) {$\mathcal{O}(\mathcal U_3)$};
\node at (16,0) (x2) {$\mathcal{O}(\mathcal U_4)$};
\node at (20,2) (x3)  {$ $};
%\node at (-6,0) (xinf) {$\invlim{\iota}{\mathcal{O}(\mathcal U_i)}$};
%\path (x1) edge[->] node[above]{$\iota'$} (x0);
%\path (x2) edge[->] node[above]{$\iota'$} (x1);
%\path (x3) edge[->] node[above]{$\iota'$} (x2);
%\path (x4) edge[->] node[above]{$\iota'$} (x3);
\node at (0,4) (y0) {$\mathcal{PO}(\mathcal U_0)$};
\node at (4,4) (y9) {$\mathcal{PO}(\mathcal U_1)$};
\node at (8,4) (y1) {$\mathcal{PO}(\mathcal U_2)$};
\node at (12,4) (y10) {$\mathcal{PO}(\mathcal U_3)$};
\node at (16,4) (y2) {$\mathcal{PO}(\mathcal U_4)$};
\node at (20,2) (y3)  {$ $};
%\node at (-6,4) (yinf) {$\invlim{\iota'}{\mathcal{PO}(\mathcal U_i')}$};
%\path (y1) edge[->] node[above]{$\iota'$} (y0);
%\path (y2) edge[->] node[above]{$\iota'$} (y1);
%\path (y3) edge[->] node[above]{$\iota'$} (y2);
%\path (y4) edge[->] node[above]{$\iota'$} (y3);
\path (y1) edge[->] node[above]{$w$} (x0);
\path (y2) edge[->] node[above]{$w$} (x1);
\path (y3) edge[dashed, ->] node[above]{$$} (x2);
\path (x0) edge[left hook->] (y0);
\path (x1) edge[left hook->] (y1);
\path (x2) edge[left hook->] (y2);
%\path (x3) edge[left hook->] (y3);
%\path (yinf) edge[->] node[left]{$w^*$} (xinf);
%\node at (-3,2) (mid) {$\simeq$};
\end{tikzpicture}

%% file: DiagonalMap3.tex
\begin{tikzpicture}[line cap=round,line join=round,x=1.0cm,y=1.0cm,scale=.5,>=latex]
%\begin{scriptsize}
%\clip(-3,-3) rectangle (3,3);
%\draw (0,0) ellipse (1 and 1);
%\node[] at (-3,-3) {$X$};
%\draw plot [->, smooth] coordinates {(0,1)(-1,0)(0,-1)(1,0)(0,1)};
%\node[fill=black, circle, inner sep=0, minimum size=3pt] at (0,0) (x0) {$$};

\node at (0,0) (x0) {$\mathcal{O}(\mathcal U_0)$};
\node at (4,0) (x1) {$\mathcal{O}(\mathcal U_2)$};
\node at (8,0) (x2) {$\mathcal{O}(\mathcal U_4)$};
\node at (12,0) (x3) {$\mathcal{O}(\mathcal U_6)$};
\node at (16,0) (x4)  {$\cdots$};
\node at (-6,0) (xinf) {$\invlim{\iota'}{\mathcal{O}(\mathcal U_{2i})}$};
\path (x1) edge[->] node[above]{$\iota'$} (x0);
\path (x2) edge[->] node[above]{$\iota'$} (x1);
\path (x3) edge[->] node[above]{$\iota'$} (x2);
\path (x4) edge[->] node[above]{$\iota'$} (x3);
\node at (0,4) (y0) {$\mathcal{PO}(\mathcal U_0)$};
\node at (4,4) (y1) {$\mathcal{PO}(\mathcal U_2)$};
\node at (8,4) (y2) {$\mathcal{PO}(\mathcal U_4)$};
\node at (12,4) (y3) {$\mathcal{PO}(\mathcal U_6)$};
\node at (16,4) (y4)  {$\cdots$};
\node at (-6,4) (yinf) {$\invlim{\iota'}{\mathcal{PO}(\mathcal U_i)}$};
\path (y1) edge[->] node[above]{$\iota'$} (y0);
\path (y2) edge[->] node[above]{$\iota'$} (y1);
\path (y3) edge[->] node[above]{$\iota'$} (y2);
\path (y4) edge[->] node[above]{$\iota'$} (y3);
\path (y1) edge[->] node[above]{$w$} (x0);
\path (y2) edge[->] node[above]{$w$} (x1);
\path (y3) edge[->] node[above]{$w$} (x2);
\path (y4) edge[->] node[above]{$w$} (x3);
\path (x0) edge[left hook->] node[above]{$ $} (y0);
\path (x1) edge[left hook->] node[above]{$ $} (y1);
\path (x2) edge[left hook->] node[above]{$ $} (y2);
\path (x3) edge[left hook->] node[above]{$ $} (y3);
\path (yinf) edge[->] node[left]{$w^*$} (xinf);
\node at (-3,2) (mid) {$\simeq$};
\end{tikzpicture}